\documentclass{amsart}
\usepackage{amssymb,amsmath,latexsym}
\usepackage{amsthm}
\usepackage{fontenc}
\usepackage{amssymb}

\numberwithin{equation}{section}

\newtheorem{theorem}{Theorem}[section]
\newtheorem{corollary}{Corollary}[theorem]
\newtheorem{lemma}[theorem]{Lemma}

\begin{document}
\author{Alexander E. Patkowski}
\title{A note on the $l$-fold Bailey Lemma and Mixed Mock Modular forms}

\maketitle
\begin{abstract}In this paper we present a method for constructing multiple-sum $q$-series for what is known as Mixed Mock Modular forms. This is one of many potential applications of a result on $l$-fold Bailey pairs given herein. We also present some multi-sum analogues of the Durfee identity, and discuss a construction of its combinatorial interpretation in terms of partitions.\end{abstract}



\section{Introduction}
The Bailey lemma is a powerful tool in the theory of $q$-series and basic hypergeometric functions for proving a wide variety of identities, including identities of Rogers--Ramanujan type and identities for false and mock theta functions, see e.g., [7],[14]. Let $(\alpha,\beta)$ be a pair of sequences, where $\alpha=(\alpha_n)_{n\ge 0}$ and $\beta=(\beta_n)_{n\ge 0}$. Then $(\alpha,\beta)$ is called a Bailey pair relative to $a$ if \[ \beta_n=\sum_{r=0}^n \frac{\alpha_r}{(aq;q)_{n+r}(q;q)_{n-r}}.\] In this definition, $(a;q)_n=(1-a)(1-aq)\cdots(1-aq^{n-1})$ is a $q$-shifted factorial. Following [5], this paper is concerned with an $l$-fold generalization of the notion of a Bailey pair, and for $(\alpha,\beta)$ a pair of sequences such that $\alpha=(\alpha_{n_1,\dots,n_l})_{n_1,\dots,n_l\in\mathbb{N}}$ and $\beta=(\beta_{n_1,\dots,n_l})_{n_1,\dots,n_l\in\mathbb{N}}$, we say that $(\alpha,\beta)$ is an $l$-fold Bailey pair relative to $a_1,\dots,a_l$ if \[ \beta_{n_1,\dots,n_l}=\sum_{r_1=0}^{n_1}\dots\sum_{r_l=1}^{n_l} \frac{\alpha_{r_1,\dots,r_l}}{\prod\limits_{i=1}^l (a_iq;q)_{n_i+r_i}(q;q)_{n_i-r_i}}.\]
The case $l=1,$ is the standard well-known Bailey pair first introduced in [7], and $(U;q)_n=(1-U)(1-Uq)\dots(1-Uq^{n-1})$ [12]. If $\alpha_{n_1,\dots,n_l}=\alpha_{n}$ when $n_1=n_2, \dots =n_l=n,$ $0$ otherwise, then $(\alpha, \beta)$ is said to be of the Joshi-Vyas type [13].
		\par Several authors [10, 14] have discussed a modular form referred to as a \it Mixed Mock Modular \rm form, which is of the form $\sum_{i\le n}M_im_i,$ where $M_i$ is a mock
		modular form, and $m_i$ is a modular form. A mixed mock modular form may be viewed as a generalization of a mock modular form. Mock modular forms, or mock theta functions, have origins dating back to Ramanujan and Watson [18]. They had discovered that these functions exhibited unique properties similar to Jacobi's classical theta functions.  \par This paper is concerned with applying results for the pair of sequences that satisfy (1.1) to obtain multi-sum $q$-series, with one type of example being expressions for products of Mock modular forms and a modular forms. See Zwegers [19] thesis for important material on the first definition of a mock theta function in relation to indefinite theta functions. For a nice overall framework on Bailey's lemma see Warnaar's paper [17].

\section{Observations on the $l$-fold Bailey Lemma} 
The purpose of this section is to outline some observations on the $l$-fold extension of Bailey's lemma which will prove fruitful in our applications to mixed mock modular forms. First, 
we recall the commonly stated relation between the two sequences in a Bailey pair $(\alpha_n, \beta_n),$

\begin{equation}\sum_{n\ge0}(\rho_1)_n(\rho_2)_n(aq/\rho_1\rho_2)^n\beta_n=\frac{(aq/\rho_1)_{\infty}(aq/\rho_2)_{\infty}}{(aq)_{\infty}(aq/\rho_1\rho_2)_{\infty}}\sum_{n\ge0}\frac{(\rho_1)_n(\rho_2)_n(aq/\rho_1\rho_2)^n\alpha_n}{(aq/\rho_1)_n(aq/\rho_2)_n}.\end{equation}
Putting $\rho_1=q^{-N_1},$ and $\rho_2=q^{-N_2},$ gives
 \begin{align} & \sum_{j\ge0}^{\min\{N_1,N_2\}}\frac{a^jq^{j^2}\alpha_j}{(aq;q)_{N_1+j}(aq;q)_{N_2+j}(q;q)_{N_1-j}(q;q)_{N_2-j}} \\ &\quad =\frac{1}{(aq;q)_{N_1+N_2}} \notag \\ & \qquad \times \sum_{j\ge0}^{\min\{N_1,N_2\}}\frac{a^jq^{j^2}}{(q;q)_{N_1-j}(q;q)_{N_2-j}}\beta_j. \notag \end{align}
 
Comparing with the definition of the $l$-fold Bailey pair we define (2.2) to be a sequence, say $\beta_{N_1, N_2},$ and we see we get a new $2$-fold Bailey pair. Iterating this idea along the $l$-fold Bailey lemma gives us the following result.

\begin{theorem} \it If $(\alpha,\beta)$ is a $l$-fold Bailey pair relative to $a_j=a,$ $j=1, 2, \cdots, l,$ then $(\bar{\alpha},\bar{\beta})$ is a $2l$-fold Bailey pair relative to $a_j=a$ where 

\begin{equation}\bar{\alpha}_{n_1, n_2, \dots, n_{2l}}=\begin{cases} q^{r_1^2+r_2^2+\dots+r_l^2}a^{r_1+r_2+\dots r_l}\alpha_{r_1, r_2, \dots, r_l},& \text {if } n_{2i-1}=n_{2i}=r_{i}, 1\le i\le l,\\ 0, & \text{otherwise,} \end{cases}\end{equation}
and
\begin{equation}\bar{\beta}_{n_1, n_2, \dots n_{2l}}=\frac{1}{\prod_{i\ge1}^{l}(aq)_{n_{2i-1}+n_{2i}}}\sum_{i_1, i_2, \dots i_l\ge0}\frac{q^{i_1^2+i_2^2+\cdots +i_l^2}a^{i_1+i_2+\cdots i_l}\beta_{i_1, i_2, \dots i_l}}{(q)_{n_1-i_1}(q)_{n_2-i_1}\cdots (q)_{n_{2l-1}-i_l}(q)_{n_{2l}-i_l}}.\end{equation}
\end{theorem}
\begin{proof} One can take the $l$-fold Bailey lemma [4, Theorem 1], and apply the same concept as (2.2). Namely, for each free variable we put $\rho_i=q^{-N_i},$ for $1\le i\le 2l,$ corresponding to the $l$-fold Bailey lemma. The theorem follows after comparison with the definition of the $l$-fold Bailey lemma provided in the introduction.  \end{proof}

\rm
The type of Bailey pairs we have constructed in Theorem 2.1 allows us relate a $2l$-fold $q$-series with a $l$-fold $q$-series. The case $l=1$ of Theorem 2.1 gives a $2$-fold Bailey pair of the Joshi-Vyas type [13]. 

\rm
In [16], Paule offered an operator proof of the special case $N_2\rightarrow\infty$ of (2.2), with $\beta_n$ expanded in its definition, and $\alpha_n$ replaced by $c_nq^{-n^2}.$ As it turns out, we can apply this observation to obtain an $l$-fold analogue of [16, Equation~(1)].

\begin{corollary} For integers $n_i,i\ge0,l\ge1,$ $i\le l,$  we have

$$\sum_{r_1,\dots ,r_l\ge0}^{n_1,\dots, n_l} \frac{c_{r_1,\dots,r_l}a^{r_1+r_2+\cdots r_l}}{\prod\limits_{i=1}^l (aq;q)_{n_i+r_i}(q;q)_{n_i-r_i}}$$
$$=\sum_{r_1, r_2, \dots ,r_l\ge0}\frac{q^{r_1^2+r_2^2+\cdots +r_l^2}a^{r_1+r_2+\cdots r_l}}{\prod\limits_{i=0}^l(q)_{n_1-r_i}}\sum_{k_1,\dots ,k_l\ge0}^{r_1,\dots, r_l} \frac{c_{k_1,\dots,k_l}q^{-k_1^2-k_2^2-\cdots -k_l^2}}{\prod\limits_{i=1}^l (aq;q)_{r_i+k_i}(q;q)_{r_i-k_i}}.$$
\end{corollary}
\begin{proof}We take to $l$-fold Bailey lemma [4, Theorem 1], and put $\rho_i=q^{-N_i},$ for $1\le i\le 2l,$ as was done in the proof of Theorem 2.1. We put $a_i=a$ in each fold of the Bailey pair as well. Then we let $N_{2j}\rightarrow\infty$ for each $j$ with $1\le 2j\le 2l.$ After expanding $\beta$ in the defintion of a $l$-fold Bailey pair we replace $\alpha_{r_1, r_2, \dots, r_l}$ with $c_{r_1, r_2, \dots, r_l} q^{-r_1^2-r_2^2-\cdots -r_l^2}.$  \end{proof}

\section{Multi-sum $q$-series and mixed mock modular forms} 
Here we will consider applications of Theorem 2.1 to mixed mock modular forms. In particular, we shall prove the following result for two third-order mock theta functions and one tenth order mock theta function using the $2$-fold Bailey lemma. The mock theta function on the right side of (3.1) is a newer function detailed in Andrews [6]. The one on the right side of (3.2) is due to Watson [18]. Lastly, the one on the right side of (3.3) has been studied by Choi [11]. 
\begin{theorem} \it We have,
\begin{equation}\sum_{j, n_1, n_2\ge0}\frac{q^{j^2+n_1^2+n_2^2+j+n_1+n_2}(-1)^j}{(q)_{n_1-j}(q)_{n_2-j}(q)_{n_1+n_2+1}(q^2;q^2)_j}=\frac{(-q)_{\infty}}{(q)_{\infty}}\sum_{n\ge0}\frac{q^{2n^2+2n}}{(-q)_{2n+1}},\end{equation}
\begin{equation}\sum_{j, n_1, n_2\ge0}\frac{(-q)_{n_1}(-q)_{n_2}q^{j^2+n_1(n_1+1)/2+n_2(n_2+1)/2+j}(-1)^j}{(q)_{n_1-j}(q)_{n_2-j}(q)_{n_1+n_2+1}(q^2;q^2)_j}=\frac{(-q)_{\infty}(-q;q^2)_{\infty}}{(q)_{\infty}(q;q^2)_{\infty}}\sum_{n\ge0}\frac{q^{n^2+n}}{(-q;q^2)_{n+1}},\end{equation}
\begin{equation}\sum_{j, n_1, n_2\ge0}\frac{(q;q^2)_{n_1}(q;q^2)_{n_2}q^{2j^2+n_1^2+n_2^2}(-1)^{n_1+n_2+j-1}(q^2;q^4)_{j-1}}{(q^2;q^2)_{n_1-j}(q^2;q^2)_{n_2-j}(q^2;q^2)_{n_1+n_2}(q^2;q^2)_{2j-1}}=\frac{(q;q^2)_{\infty}^2(q^4;q^4)_{\infty}}{(q^2;q^2)_{\infty}^2(-q^4;q^4)_{\infty}}\psi(q^4),\end{equation}
where $\psi(q)=\sum_{n\ge1}q^{n(n+1)/2}/(q;q^2)_n.$
\end{theorem}

\begin{proof} First recall [5] that
\begin{align} & \sum_{n_1\ge0}^{\infty}\sum_{n_2\ge0}^{\infty}(x)_{n_1}(y)_{n_1}(z)_{n_2}(w)_{n_2}(a_1q/xy)^{n_1}(a_2q/zw)^{n_2}\beta_{n_1, n_2} \\ &\quad =\frac{(a_1q/x)_{\infty}(a_1q/y)_{\infty}(a_2q/z)_{\infty}(a_2q/w)_{\infty}}{(a_1q)_{\infty}(a_1q/xy)_{\infty}(a_2q)_{\infty}(a_2q/zw)_{\infty}} \notag \\ & \qquad \times \sum_{n_1\ge0}^{\infty}\sum_{n_2\ge0}^{\infty}\frac{(x)_{n_1}(y)_{n_1}(z)_{n_2}(w)_{n_2}(a_1q/xy)^{n_1}(a_2q/zw)^{n_2}\alpha_{n_1, n_2}}{(a_1q/x)_{n_1}(a_1q/y)_{n_1}(a_2q/z)_{n_2}(a_2q/w)_{n_2}}. \notag \end{align}

\bf{The case (3.1):} \rm We begin by recalling the following Bailey pair relative to $q$ [15, Equation~(2.19)--(2.20)] 
\begin{align*} \alpha_n&=\frac{q^{n^2}(1-q^{2n+1})}{(1-q)}\sum_{|j|\le n}(-1)^jq^{-j^2} \\ \beta_n &= \frac{(-1)^n}{(q^2;q^2)_n} \end{align*}
into (2.2) and insert the resulting $2$-fold pair

\begin{equation}\alpha_{N_1, N_2}=\begin{cases} \frac{q^{2n^2+n}(1-q^{2n+1})}{(1-q)}\sum_{|j|\le n}(-1)^jq^{-j^2},& \text {if } N_1=N_2=n\\ 0, & \text{otherwise,} \end{cases}\end{equation}
and
\begin{equation}\beta_{N_1, N_2}=\frac{1}{(q)_{N_{1}+N_{2}+1}}\sum_{j\ge0}\frac{q^{j^2+j}(-1)^{j}}{(q)_{N_1-j}(q)_{N_2-j}(q^2;q^2)_{j}}.\end{equation}

into (3.4) with $x, y, z, w\rightarrow\infty,$ $a_1=a_2=q$ to get
\begin{equation}\sum_{j, n_1, n_2\ge0}\frac{q^{j^2+n_1^2+n_2^2+j+n_1+n_2}(-1)^j}{(q)_{n_1-j}(q)_{n_2-j}(q)_{n_1+n_2+1}(q^2;q^2)_j}=\frac{1}{(q)_{\infty}^2}\sum_{n\ge0}q^{4n^2+3n}(1-q^{2n+1})\sum_{|j|\le n}(-1)^jq^{-j^2}.\end{equation}
Note that the indefinite quadratic series on the right side of (3.7) is found to be related to a mock theta function noted by Andrews by [6, Equation~(1.15)]
$$\sum_{n\ge0}\frac{q^{2n^2+2n}}{(-q)_{2n+1}}=\frac{1}{(q^2;q^2)_{\infty}}\sum_{n\ge0}q^{4n^2+3n}(1-q^{2n+1})\sum_{|j|\le n}(-1)^jq^{-j^2}.$$ Hence, replacing the indefinite quadratic form series on the right side of (3.7) with this mock theta function after being multiplied by $(q^2;q^2)_{\infty}$ gives the identity.\newline
\bf{The case (3.2):} \rm A simple consequence of the second pair in table one in [3, pg.75] inserted in a limiting case of Bailey's lemma ($\rho_1=-q,$ $\rho_2\rightarrow\infty,$ in (2.1)), is given as
\begin{equation}\sum_{n\ge0}\frac{q^{n(n+1)}}{(-q;q^2)_{n+1}}=\frac{(-q^2;q^2)_{\infty}}{(q^2;q^2)_{\infty}}\sum_{n\ge0}q^{3n^2+2n}(1-q^{2n+1})\sum_{|j|\le n}(-1)^jq^{-j^2}.\end{equation}
(The $q$-series on the left side is a third-order mock theta function [18].) Inserting (3.5)--(3.6) into (3.4) with $x=z=-q, y, w\rightarrow\infty,$ $a_1=a_2=q,$ gives the result after invoking (3.8). \newline
\bf{The case (3.3):} \rm First we note a result due to Choi [11, page~508]
 \begin{align*} &\sum_{n\ge1}\frac{q^{n(n+1)/2}}{(q;q^2)_n} \\ &\quad =\frac{(-q)_{\infty}}{(q)_{\infty}} \sum_{n\ge0}q^{5n^2+4n+1}(1-q^{2n+1})\sum_{|j|\le n}q^{-j^2}\\ & \qquad - \frac{(-q)_{\infty}}{(q)_{\infty}} \sum_{n\ge1}q^{5n^2-n}(1-q^{2n})\sum_{j=-n}^{n-1}q^{-j^2-j}. \end{align*}
This follows from the pair [2, page~131, Equation~(7.18)--(7.19)] inserted into the $\rho_1=-1,$ $\rho_2\rightarrow\infty$ case of (2.1). Inserting the pair due to Bringmann and Kane [9, Theorem 2.3, (1)] (where $\beta_0=0$)
\begin{align*} \alpha_{2n}&=q^{2n^2-2n}(1-q^{4n})\sum_{j=-n}^{n-1}q^{-2j^2-2j}, \\
\alpha_{2n+1}&=-q^{2n^2}(1-q^{4n+2})\sum_{|j|\le n}(-1)^jq^{-2j^2}, \\
 \beta_n &= \frac{(-1)^n(q;q^2)_{n-1}}{(q)_{2n-1}}, \end{align*}

 into (2.2) and then inserting the resulting pair into (3.4) with $x=z=-q^{1/2}, y, w\rightarrow\infty,$ $a_1=a_2=1,$ gives
 
 \begin{align*} &\sum_{\substack{j,n_1,n_2\ge 0 \\ n_1-j,n_2-j\ge 0}}\frac{(-q^{1/2})_{n_1}(-q^{1/2})_{n_2}q^{j^2+n_1^2/2+n_2^2/2}(-1)^{j}(q;q^2)_{j-1}}{(q)_{n_1-j}(q)_{n_2-j}(q)_{n_1+n_2}(q)_{2j-1}} \\ &\quad =\frac{(-q^{1/2};q)_{\infty}^2}{(q;q)_{\infty}^2} \sum_{n\ge1}q^{10n^2-2n}(1-q^{4n})\sum_{j=-n}^{n-1}q^{-2j^2-2j} \\ & \qquad - \frac{(-q^{1/2};q)_{\infty}^2}{(q;q)_{\infty}^2} \sum_{n\ge0}q^{10n^2+8n+2}(1-q^{4n+2})\sum_{|j|\le n}q^{-2j^2}. \end{align*}
Hence comparing these indefinite theta functions with Choi's indefinite theta representation and a little manipulation gives (3.3).
\end{proof}
\section{Multi-sum identities for products} 
 \par We provide some further identities, namely of the sum-to-product type, including multi-dimensional analogues of the Durfee square identity. First, it should be noted the observation made in the second section may be applied in other directions. For example, the identity
(2.2) may be used to obtain many interesting $q$-polynomial identities. The choice [5, Equation~(2.9)--(2.10)] $\alpha_0=1,$ $\alpha_n=(-1)^nq^{n(3n-1)/2}(1+q^n),$ $n>0,$ and $\beta_n=1/(q)_n,$ in (2.2) gives us a key identity used in [4, Equation~(5.33)], 
$$\sum_{j\in\mathbb{Z}}\frac{(-1)^jq^{j(5j+1)/2}}{(q)_{L_1-j}(q)_{L_1-j}(q)_j^2}=\frac{1}{(q)_{L_1+L_2}}\sum_{n\ge0}\frac{q^{n^2}}{(q)_{L_1-n}(q)_{L_2-n}(q)_n},$$
which may in turn be used to obtain [4, Equation~(5.30)] 
$$\sum_{n_1,n_2\ge0}\frac{q^{n_1^2+n_1n_2+n_2^2}}{(q)_{L_1-n_1}(q)_{L_2-n_2}(q)_{n_1}(q)_{n_2}(q)_{n_1+n_2}}=\sum_{n\ge0}\frac{q^{n^2}}{(q)_{L_1+L_2}(q)_{L_1-n}(q)_{L_2-n}(q)_n},$$
after appealing to the $l=2$ case of Theorem 2.1 with a $2$-fold Bailey pair obtained by Joshi and Vyas [13] (with $\alpha_{0,0}=1,$ $\alpha_{n_1,n_2}=0$ if $n_1\neq n_2$)

$$\begin{aligned}\alpha_{n_1,n_2}&=(-1)^nq^{n(n-1)/2}(1+q^n), \\ &\quad \text{if $n=n_1=n_2,$ and}\\ \beta_{n_1,n_2} &= \frac{q^{n_1n_2}}{(q)_{n_1}(q)_{n_2}(q)_{n_1+n_2}}. \end{aligned}$$

 Another direct corollary of [5, Equation~(2.9)--(2.10)] leads us to the Rogers--Ramanujan type identity (by using [5, Corollary 1, $s=2$], Theorem 2.1 with $l=1,$ and Jacobi's triple product identity), 
\begin{equation}\sum_{j,n_1,n_2\ge0}\frac{q^{j^2+n_1^2+n_2^2}}{(q)_{n_1+n_2}(q)_{n_1-j}(q)_{n_2-j}(q)_j}=\frac{(q,q^8,q^9;q^9)_{\infty}}{(q)_{\infty}^2},\end{equation}
which can be shown to be related to one of the four modulus 9 Andrews--Gordon identities.
Second, to see how one may apply Theorem 2.1 in its full generality, we may consider the pair $\alpha_0=1,$ $0$ otherwise, and $\beta_n=1/(q)_{n}^2.$ Using this pair with (2.2) gives us
\begin{equation} \frac{1}{(q)_{n_1}^2(q)_{n_2}^2}=\frac{1}{(q)_{n_1+n_2}}\sum_{j\ge0}\frac{q^{j^2}}{(q)_{n_1-j}(q)_{n_2-j}(q)_{j}^2},\end{equation}
which is equivalent to the $2$-fold Bailey pair 

\begin{equation}\alpha_{n_1, n_2}=\begin{cases} 1,& \text {if } n_1=n_2=0,\\ 0, & \text{otherwise,} \end{cases}\end{equation}
and
\begin{equation}\beta_{n_1, n_2}=\frac{1}{(q)_{n_1+n_2}}\sum_{j\ge0}\frac{q^{j^2}}{(q)_{n_1-j}(q)_{n_2-j}(q)_{j}^2}.\end{equation}
Applying this pair to the $l=2$ case of Theorem 2.1 gives us 
\begin{equation}\alpha_{n_1, n_2,n_3,n_4}=\begin{cases} 1,& \text {if } n_1=n_2=n_3=n_4=0,\\ 0, & \text{otherwise,} \end{cases}\end{equation}
and
\begin{align} & \beta_{n_1, n_2, n_3, n_4} \\ &\quad =\frac{1}{(q)_{n_1+n_2}(q)_{n_3+n_4}} \notag \\ & \qquad \times \sum_{i_1,i_2,j\ge0}\frac{q^{i_1^2+i_2^2+j^2}}{(q)_{n_1-i_1}(q)_{n_2-i_1}(q)_{n_3-i_2}(q)_{n_4-i_2}(q)_{i_1+i_2}(q)_{i_1-j}(q)_{i_2-j}(q)_{j}^2}. \notag \end{align}

Therefore we may use the $l$-fold Bailey lemma to give us
\begin{equation}\sum_{n_1,n_2,j\ge0}\frac{q^{j^2+n_1^2+n_2^2}}{(q)_{n_1+n_2}(q)_{n_1-j}(q)_{n_2-j}(q)_{j}^2}=\frac{1}{(q)_{\infty}^2},\end{equation}
and

\begin{equation}\sum_{n_1, n_2, n_3, n_4,i_1,i_2,j\ge0}\frac{q^{n_1^2+n_2^2+n_3^2+n_4^2+i_1^2+i_2^2+j^2}}{(q)_{n_1+n_2}(q)_{n_3+n_4}(q)_{n_1-i_1}(q)_{n_2-i_1}(q)_{n_3-i_2}(q)_{n_4-i_2}(q)_{i_1+i_2}(q)_{i_1-j}(q)_{i_2-j}(q)_{j}^2}=\frac{1}{(q)_{\infty}^4}.\end{equation}
Continuing this idea leads us to an expression for $1/(q)_{\infty}^{2M},$ $M\in\mathbb{N}.$
\par We give a combinatorial interpretation of (4.7). Let the $q$-binomial coefficients be denoted as
$$\binom{n}{k}_q:=\frac{(q)_n}{(q)_k(q)_{n-k}}.$$
Recall that a partition of $n$ is the number of ways to write $n$ as a sum of smaller numbers, which are called parts. The Ferrers graph is the display of the partition in which each part is represented as a row of nodes [8, Definition 2.2]. The Durfee square of a partition is the largest square of nodes in the upper left hand corner of the Ferrers graph [8, Definition 2.3]. We will need a Lemma from Bressoud and Zeilberger, as well as their following details related to their proof.
\begin{lemma} ([8, Lemma 2.5]) let $m_1\ge m_2\ge \dots m_k\ge0,$ $k\ge1,$ then
$$\frac{q^{m_1^2+m_2^2+\cdots +m_k^2}}{(q)_{m_1}} \binom{m_1}{m_2}_q\binom{m_{2}}{m_{3}}_q\cdots \binom{m_{k-1}}{m_{k}}_q=\sum_{m\ge0}D_k(m)q^m,$$
where $D_k(m)$ is the number of partitions of $m$ such that for $1\le i\le k,$ the $i$-th Durfee square of the Ferrers graph of the partition is an $m_i\times m_i$ square, and there are no parts below the $k$-th Durfee square. 
\end{lemma}
In the details which follow in [8, pg.45--46] of this lemma, it is noted that $q^{m_1^2+m_2^2+\cdots +m_k^2}$ generates $k$ Durfee squares. The size of the Durfee square corresponding to the partition generated by $q^{m_k^2}$ is clearly $m_k.$ Therein, it is also noted that in constructing the partition, for $k\ge2$ the parts generated by the component $\binom{m_{k-1}}{m_{k}}_q$ are placed on the right side of the Durfee square of size $m_k.$ These are partitions with parts bounded in magnitude by $m_{k-1}-m_k,$ and number of parts bounded by $m_k.$ Returning to (4.7), we note that we may write
 \begin{equation}\frac{q^{j^2+n_1^2+n_2^2}}{(q)_{n_1+n_2}(q)_{n_1-j}(q)_{n_2-j}(q)_{j}^2}=\frac{q^{j^2+n_1^2+n_2^2}}{(q)_{n_1+n_2}(q)_{n_1}(q)_{n_2}}\binom{n_1}{j}_q\binom{n_2}{j}_q. \end{equation}
We examine the right side of (4.9) in three components,
 \begin{equation} \frac{1}{(q)_{n_1+n_2}}\times \frac{q^{n_1^2+j^2}}{(q)_{n_1}}\binom{n_1}{j}_q\times \frac{q^{n_2^2}}{(q)_{n_2}}\binom{n_2}{j}_q. \end{equation}
 Now we say a partition pair of $n$ is a pair of partitions $(\pi_1,\pi_2)$ where the sum of both their respective parts is $n.$ For example, the right side of (4.7) generates the number of partition pairs, where both $\pi_i,$ $i=1,2,$ have no restriction on their parts. Similarly, a partition triple of $n$ is a triple of partitions $(\pi_1,\pi_2, \pi_3)$ where the sum of their parts is $n.$ The middle component of (4.10) is Lemma 1 with $k=2,$ $m_1=n_1,$ $m_2=j.$ The third component in (4.10) generates the same partition as the second, but with the condition that the smallest Durfee square $j\times j$ is removed. Consequently, we place the parts that were originally to the right of the $j\times j$ Durfee square below the remaining Durfee square. Observe the first component of (4.10) generates partitions into at most $n_1+n_2$ parts. Let $s(\pi)$ denote the size of the largest Durfee square of $\pi.$ 
 \par We combine these observations to interpret (4.10) as the generating function for a partition triple $(\pi_1,\pi_2, \pi_3)$ of $n$ where: (1) $\pi_1$ is the number of partitions into at most $s(\pi_2)+s(\pi_3)$ parts. (2) $\pi_2$ is the number of partitions such that there are two Durfee squares, $j\times j$ being the smallest Durfee square of the Ferrars graph, and there are no parts below this square. (3) $\pi_3$ is the number of partitions such that there is one Durfee square which is greater than or equal to the smallest Durfee square of $\pi_2,$ and not necessarily equal to the largest Durfee square of $\pi_2.$ This is because only $j$ is bounded by $\min\{n_1,n_2\}$ in the triple sum. Parts below the Durfee square are bounded in magnitude by $s(\pi_3)-j,$ and number of parts bounded by $j.$ Summing over $j,n_1,n_2\ge0$ gives the left side of (4.7). In comparison with (4.8), it would seem a similar, albeit more complicated, interpretation may be offered for the $2M$-th power of the generating function for the number of partitions of $n.$

1390 Bumps River Rd. \\*
Centerville, MA
02632 \\*
USA \\*
\\*
ul. A. E. Ody\'{n}ca 47 \\*
02-606 Warsaw\\*
Poland\\*
E-mail: alexpatk@hotmail.com, alexepatkowski@gmail.com
\end{document}